\documentclass[12pt,a4paper,UKenglish,thmsb]{article}
\usepackage{graphicx}
\usepackage{amsfonts}
\usepackage{amssymb}
\usepackage{amsmath}
\usepackage{makeidx}

\newtheorem{theorem}{Theorem}[section]

\newtheorem{condition}{Condition}

\newtheorem{corollary}[theorem]{Corollary}

\newtheorem{definition}{Definition}[section]

\newtheorem{lemma}[theorem]{Lemma}
\newtheorem{notation}{Notation}

\newtheorem{proposition}[theorem]{Proposition}
\newtheorem{remark}{Remark}[section]

\newenvironment{proof}[1][Proof]{\textbf{#1.} }{\ \rule{0.5em}{0.5em}}

\begin{document}

\title{Upper bounds for transition probabilities on graphs and isoperimetric
inequalities}
\author{Andr\'{a}s Telcs \\
{\small Department of Computer Science and Information Theory, }\\
{\small University of Technology and Economy Budapest}\\
{\small telcs@szit.bme.hu}}
\maketitle

\begin{abstract}
In this paper necessary and sufficient conditions are presented for heat
kernel upper bounds for random walks on weighted graphs. \ Several
equivalent conditions are given in the form of isoperimetric inequalities.


\end{abstract}

\tableofcontents

\section{Introduction}

\setcounter{equation}{0}

Heat kernel upper bounds are subject of heavy investigations for decades. \
Aronson, Moser, Varopoulos, Davies, Li and Yau, Grigor'yan, Saloff-Coste and
others contributed to the development of the area (for the history see the
bibliography of \cite{VSC}). \ The work of Varopoulos highlighted the
connection between the heat kernel upper estimates and isoperimetric
inequalities. \ The present paper follows this approach and provides
transition probability upper estimates of reversible Markov chains in a
general form under necessary and sufficient conditions. \ The conditions are
isoperimetric inequalities which control the smallest Dirichlet eigenvalue,
the capacity or the mean exit time of a finite vertex set. \ In addition,
the paper presents a generalization of the Davies-Gaffney inequality (c.f. 
\cite{CG}) which is a tool in the proof of the off-diagonal upper estimate.

Let us consider a countable infinite connected graph $\Gamma $.

\begin{definition}
A symmetric weight function $\mu _{x,y}=\mu _{y,x}>0$ is given on the edges $%
x\sim y$. This weight function induces a measure $\mu (x)$%
\begin{eqnarray*}
\mu (x) &=&\sum_{y\sim x}\mu _{x,y}, \\
\mu (A) &=&\sum_{y\in A}\mu (y).
\end{eqnarray*}%
The graph is equipped with the usual (shortest path length) graph distance $%
d(x,y)$ and open metric balls are defined for $x\in \Gamma ,$ $R>0$ as 
\begin{equation*}
B(x,R)=\{y\in \Gamma :d(x,y)<R\}
\end{equation*}%
and its $\mu -$measure is denoted by $V(x,R)$%
\begin{equation*}
V\left( x,R\right) =\mu \left( B\left( x,R\right) \right) .
\end{equation*}%
The weighted graph has the volume doubling$\ $ property \textbf{$(\mathbf{VD)%
}$} if there is a constant $D_{V}>0$ such that for all $x\in \Gamma $ and $%
R>0$%
\begin{equation}
V(x,2R)\leq D_{V}V(x,R).  \label{VD}
\end{equation}
\end{definition}

\begin{definition}
The edge weights define a reversible Markov chain $X_{n}\in \Gamma $, i.e. a
random walk on the weighted graph $(\Gamma ,\mu )$ with transition
probabilities 
\begin{align*}
P(x,y)& =\frac{\mu _{x,y}}{\mu (x)}, \\
P_{n}\left( x,y\right) & =\mathbb{P}(X_{n}=y|X_{0}=x).
\end{align*}%
The \textquotedblright heat kernel\textquotedblright\ of the random walk is\ 
\begin{equation*}
p_{n}(x,y)=p_{n}\left( y,x\right) =\frac{1}{\mu \left( y\right) }P_{n}\left(
x,y\right) .
\end{equation*}
\end{definition}

Let $\mathbb{P}_{x},\mathbb{E}_{x}$ denote the probability measure and
expected value with respect to the Markov chain $X_{n}$ \ if $X_{0}=x.$\ 

\begin{definition}
The Markov operator $P$ \ of the reversible Markov chain is naturally
defined by%
\begin{equation*}
Pf\left( x\right) =\sum P\left( x,y\right) f\left( y\right) .
\end{equation*}
\end{definition}

\begin{definition}
The Laplace operator on the weighted graph $\left( \Gamma ,\mu \right) $ is
defined simply as 
\begin{equation*}
\Delta =P-I.
\end{equation*}
\end{definition}

\begin{definition}
For $A\subset \Gamma $ consider $P^{A}$ the Markov operator $P$ restricted
to $A.$ This operator is the Markov operator of the killed Markov chain,
which is killed on leaving $A,$ also corresponds to the Dirichlet boundary
condition on $A$. \ Its iterates are denoted by $P_{k}^{A}$.
\end{definition}

\begin{definition}
The Laplace operator with Dirichlet boundary conditions on a finite set $%
A\subset \Gamma $ is defined as%
\begin{equation*}
\Delta ^{A}f\left( x\right) =\left\{ 
\begin{array}{ccc}
\Delta f\left( x\right) & \text{if} & x\in A \\ 
0 & if & x\notin A%
\end{array}%
\right. .
\end{equation*}%
The smallest eigenvalue of $-\Delta ^{A}$ is denoted in general by $\lambda
(A)$ and for $A=B(x,R)$ it is denoted by $\lambda =\lambda (x,R)=\lambda
(B(x,R)).$
\end{definition}

\begin{definition}
On the weighted graph $\left( \Gamma ,\mu \right) $ the inner product is
defined as 
\begin{equation*}
\left( f,g\right) =\left( f,g\right) _{\mu }=\sum_{x\in \Gamma }f\left(
x\right) g\left( x\right) \mu \left( x\right) .
\end{equation*}
\end{definition}

\begin{definition}
The energy or Dirichlet form $\mathcal{E}\left( f,f\right) $ associated to
the Laplace operator $\Delta $ is defined as 
\begin{equation*}
\mathcal{E}\left( f,f\right) =-\left( \Delta f,f\right) =\frac{1}{2}%
\sum_{x,y\in \Gamma }\mu _{x,y}\left( f\left( x\right) -f\left( y\right)
\right) ^{2}.
\end{equation*}
\end{definition}

Using this notation the smallest eigenvalue of $-\Delta ^{A}$ can be defined
by 
\begin{equation}
\lambda \left( A\right) =\inf \left\{ \frac{\mathcal{E}\left( f,f\right) }{%
\left( f,f\right) }:f\in c_{0}\left( A\right) ,f\neq 0\right\}  \label{ldef}
\end{equation}%
as well.

The exit time from a set $A\subset \Gamma $ is 
\begin{equation*}
T_{A}=\min \{k\geq 0:X_{k}\in \Gamma \backslash A\}
\end{equation*}%
and its expected value is denoted by 
\begin{equation*}
E_{x}(A)=\mathbb{E}(T_{A}|X_{0}=x)
\end{equation*}%
and we will use the short notations $E=E(x,R)=E_{x}(x,R)=E_{x}\left( B\left(
x,R\right) \right) $.

The main task of this paper is to find estimates of the heat kernel. \ Such
estimates have a vast literature (see the bibliography of \cite{CG1} as a
starting point).

The diagonal upper estimate 
\begin{equation*}
p_{n}\left( x,x\right) \leq Cn^{-\gamma }
\end{equation*}%
is equivalent to the Faber-Krahn inequality%
\begin{equation*}
\lambda ^{-1}\left( A\right) \leq C\mu \left( A\right) ^{\delta }\text{ for
all }A\subset \Gamma
\end{equation*}%
for some $\gamma ,\delta ,C>0$ (c.f. \cite{Carr},\cite{Gfk}).

The classical off-diagonal upper estimate has the form 
\begin{equation*}
p_{n}\left( x,x\right) \leq \frac{C_{d}}{n^{\frac{d}{2}}}\exp \left[ -\frac{%
d^{2}\left( x,y\right) }{2n}\right]
\end{equation*}%
for the random walk on the integer lattice $\mathbb{Z}^{d}$, which reflects
the basic fact that%
\begin{equation*}
E\left( x,R\right) \simeq R^{2}.
\end{equation*}%
Here and in the whole sequel $c,C$ will denote unimportant constants, their
values may change from place to place. Coulhon and Grigor'yan \cite{CG}
proved for random walks on weighted graphs that the relative Faber-Krahn
inequality 
\begin{equation*}
\lambda ^{-1}\left( A\right) \leq CR^{2}\left( \frac{\mu \left( A\right) }{%
V\left( x,R\right) }\right) ^{\delta }\text{ for all }A\subset B\left(
x,R\right) ,x\in \Gamma ,R>0
\end{equation*}%
is equivalent to the conjunction of the volume doubling property $\left( \ref%
{VD}\right) $ and$\ $%
\begin{equation*}
p_{n}\left( x,y\right) \leq \frac{C}{V\left( x,\sqrt{n}\right) }\exp \left[
-c\frac{d^{2}\left( x,y\right) }{n}\right] .
\end{equation*}

In the last fifteen years several works were devoted to the study of
sub-diffusive behavior of fractals, which typically means that the condition 
$\left( E_{\beta }\right) $ 
\begin{equation}
E\left( x,R\right) \simeq R^{\beta }  \label{ebeta}
\end{equation}%
for a $\beta >2\ $is satisfied. \ On particular fractals it was possible to
show that \ the following heat kernel upper bound $\left( UE_{\beta }\right) 
$ holds:%
\begin{equation}
p_{t}\left( x,y\right) \leq \frac{C}{V\left( x,t^{\frac{1}{\beta }}\right) }%
\exp \left[ -c\left( \frac{R^{\beta }}{t}\right) ^{\frac{1}{\beta -1}}\right]
.  \label{UEbb}
\end{equation}
Grigor'yan has shown in \cite{G1}\ that in continuous settings under the
volume doubling condition $\left( UE_{\beta }\right) $ is equivalent to the
conjunction of $\left( E_{\beta }\right) $ and%
\begin{equation*}
\lambda ^{-1}\left( A\right) \leq CR^{\beta }\left( \frac{\mu \left(
A\right) }{V\left( x,R\right) }\right) ^{\delta }\text{ for all }A\subset
\Gamma ,x\in \Gamma ,R>0,
\end{equation*}%
The upper estimate $\left( UE_{\beta }\right) $ has been shown for several
particular fractals prior to \cite{G1} (see the literature in \cite{K1} or
for very recent ones in \cite{G1}, \cite{K}, \cite{GT2} or \cite{Tfull}) and
generalized to some class of graphs in \cite{TD} and \cite{Tfull}. In \cite%
{GT2} an example is given for a graph which satisfies $\left( UE_{\beta
}\right) $ and the lower counterpart (differing only in the constants $C,c$%
). This example is an easy modification of the Vicsek tree .
\bigskip 

One should put increasing weights on the edges of increasing blocks of the
tree .
It is easy to see that on this tree the volume doubling condition and $%
\left( E_{\beta }\right) $ holds. Another construction based on the Vicsek
tree is the stretched Vicsek tree, which is given in \cite%
{Tfull} and it violates $\left( E_{\beta }\right) $ while it satisfies $%
\left( VD\right) $. It can be obtained by replacing the edges of the
consecutive block of the tree with paths of slowly increasing length. 

It was shown in \cite{Tfull} that this example is not covered by any earlier
results but satisfies enough regularity properties to obtain a heat kernel
upper estimate which is local not only in the volume but in the mean exit
time as well. We shall return to this example briefly in Section \ref%
{sexample}.

The main result of the present paper gives equivalent isoperimetric
inequalities which imply on- and off-diagonal upper estimates in a general
form. Let us give here only one, the others will be stated after the
necessary definitions.

The result states among others that if there are $C,\delta >0$ such that for
all $x\in \Gamma ,R,n>0$ if for all $A\subset B\left( x,3R\right) ,B=B\left(
x,R\right) ,2B=B\left( x,2R\right) $%
\begin{equation*}
\lambda ^{-1}\left( A\right) \leq CE\left( x,R\right) \left( \frac{\mu
\left( A\right) }{\mu \left( 2B\right) }\right) ^{\delta }
\end{equation*}%
holds, then the (local) diagonal upper estimate $\left( DUE\right) $ holds:
there is a $C>0,$ such that for all $x\in \Gamma ,n>0$%
\begin{equation}
p_{n}\left( x,x\right) \leq \frac{C}{V\left( x,e\left( x,n\right) \right) },
\tag{DUE}  \label{DUE1}
\end{equation}%
and the (local) upper estimate $\left( UE\right) $ holds: there are $%
c,C>0,\beta >1$ such that for all $x,y\in \Gamma ,n>0$%
\begin{equation}
p_{n}\left( x,y\right) \leq \frac{C}{V\left( x,e\left( x,n\right) \right) }%
\exp \left[ -c\left( \frac{E\left( x,d\left( x,y\right) \right) }{n}\right)
^{^{\frac{1}{\beta -1}}}\right] .  \tag{$UE$}  \label{UE1}
\end{equation}%
Here $e\left( x,n\right) $ is the inverse of $E\left( x,R\right) $ in the
second variable. \ The existence follows easily from the strong Markov
property (c.f. \cite{tER}). The full result contains the corresponding
reverse implications as well.

The presented results are motivated by the work of Kigami \cite{K} and
Grigor'yan \cite{G1}. Those provide necessary and sufficient conditions for
the case when $E\left( x,R\right) \simeq R^{\beta }$ \ uniformly in the
space (they work in the continuous settings on measure metric spaces). Our
result is an adaptation to the discrete settings and generalization of the
mentioned works relaxing the condition on the mean exit time. It seems that
the results carry over to the continuous setup without major changes
provided the stochastic process has some natural properties (which among
others imply that it has continuous heat kernel, c.f. \cite{G1}).

The structure of the paper is the following. \ In Section \ref{sdefs} we lay
down the necessary definitions and give the statement of the main results.
In Section \ref{sineq} some potential theoretical inequalities are collected
and equivalence of the isoperimetric inequalities are given. In Section \ref%
{sfk} the proof of the main result is presented. \ Finally Section \ref%
{sexample} provides further details of the example of the stretched Vicsek
tree.\ 

\section{Basic definitions and the results}

\setcounter{equation}{0}$\label{sdefs}$We consider the weighted graph $%
\left( \Gamma ,\mu \right) $ as it was introduced in the previous section.

\begin{condition}
In many statements we assume that condition $\mathbf{(p}_{0}\mathbf{)}$
holds, that is there is an universal $p_{0}>0$ such that for all $x,y\in
\Gamma ,x\sim y$ 
\begin{equation}
\frac{\mu _{x,y}}{\mu (x)}\geq p_{0},  \label{p0}
\end{equation}
\end{condition}

\begin{notation}
The following standard notations will be used.%
\begin{equation*}
\left\Vert f\right\Vert _{1}=\sum_{x\in \Gamma }\left\vert f\left( x\right)
\right\vert \mu \left( x\right)
\end{equation*}%
and%
\begin{equation*}
\left\Vert f\right\Vert _{2}=\left( f,f\right) ^{1/2}.
\end{equation*}
\end{notation}

\begin{definition}
\label{defG}We introduce 
\begin{equation*}
G^{A}(y,z)=\sum_{k=0}^{\infty }P_{k}^{A}(y,z)
\end{equation*}%
the local Green function, the Green function of the killed walk and the
corresponding Green kernel as 
\begin{equation*}
g^{A}(y,z)=\frac{1}{\mu \left( z\right) }G^{A}(y,z).
\end{equation*}
\end{definition}

\begin{definition}
Let $\partial A$ denote the boundary of a set $A\subset \Gamma :$ $\partial
A=\{z\in \Gamma \backslash A:z\sim y\in A$ $\}.$ The closure of $A$ will be
denoted by $\overline{A}$ and defined by $\overline{A}=A\cup \partial A$,
also let $A^{c}=\Gamma \backslash A.$
\end{definition}

\begin{notation}
For two real series $a_{\xi },b_{\xi },\xi \in S$ we shall use the notation $%
a_{\xi }\simeq b_{\xi }$ \ if there is a $C>1$ such that for all $\xi \in S$%
\begin{equation*}
C^{-1}a_{\xi }\leq b_{\xi }\leq Ca_{\xi }.
\end{equation*}
\end{notation}

For convenience we introduce a short notation for the volume of the annulus $%
B\left( x,R\right) \backslash B\left( x,r\right) $ for $R>r>0$: 
\begin{equation*}
v(x,r,R)=V(x,R)-V(x,r).
\end{equation*}

\begin{definition}
The extreme mean exit time is defined as 
\begin{equation*}
\overline{E}(A)=\max_{x\in A}E_{x}(A)
\end{equation*}%
and the $\overline{E}(x,R)=\overline{E}(B(x,R))$ simplified notation will be
used.
\end{definition}

\begin{definition}
We say that the graph satisfies condition $\left( \overline{E}\right) $ if
there is a $C>0$ such that for all $x\in \Gamma ,R>0$%
\begin{equation*}
\overline{E}\left( x,R\right) \leq CE\left( x,R\right) .
\end{equation*}
\end{definition}

\begin{definition}
We will say that the weighted graph $(\Gamma ,\mu )$ satisfies the time
comparison principle $\left( \mathbf{TC}\right) $ if there is a constant $%
C>1 $ such that for all $x\in \Gamma $ and $R>0,y\in B\left( x,R\right) $%
\begin{equation}
\frac{E(y,2R)}{E\left( x,R\right) }\leq C.  \label{TC}
\end{equation}
\end{definition}

\begin{remark}
It is clear that $\left( TC\right) $ implies $\left( \overline{E}\right) .$
\end{remark}

\begin{definition}
For any two disjoint sets, $A,B\subset \Gamma ,$ the resistance between them 
$\rho (A,B)$ is defined as 
\begin{equation}
\rho (A,B)=\left( \inf \left\{ \mathcal{E}\left( f,f\right)
:f|_{A}=1,f|_{B}=0\right\} \right) ^{-1}  \label{resdef}
\end{equation}%
and we introduce 
\begin{equation*}
\rho (x,r,R)=\rho (B(x,r),\Gamma \backslash B(x,R))
\end{equation*}%
for the resistance of the annulus about $x\in \Gamma ,$ with $R>r>0$.
\end{definition}

\begin{theorem}
\label{tmain}Assume that $(\Gamma ,\mu )$ satisfies $(p_{0})$. Then the
following inequalities are equivalent ( assuming that each statement
separately holds for all $x,y\in \Gamma ,$ $R>0,$ $n>0,$ $D\subset A\subset
B=B\left( x,3R\right) $ with fixed independent $\delta ,C>0,\beta >1$) .%
\begin{equation}
\overline{E}(A)\leq CE\left( x,R\right) \left( \frac{\mu \left( A\right) }{%
\mu \left( B\right) }\right) ^{\delta },  \tag{E}  \label{sFKE}
\end{equation}%
\begin{equation}
\lambda (A)^{-1}\leq CE\left( x,R\right) \left( \frac{\mu \left( A\right) }{%
\mu \left( B\right) }\right) ^{\delta },  \tag{FK}  \label{sFKll}
\end{equation}%
\begin{equation}
\rho (D,A)\mu \left( D\right) \leq CE\left( x,R\right) \left( \frac{\mu
\left( A\right) }{\mu \left( B\right) }\right) ^{\delta },  \tag{$\rho $}
\label{sFKrr}
\end{equation}%
\newline
\begin{equation}
p_{n}(x,x)\leq \frac{C}{V\left( x,e\left( x,n\right) \right) }\text{ } 
\tag{DUE}  \label{DUE}
\end{equation}%
together with $\left( VD\right) $ and $\left( TC\right) ,$%
\begin{equation}
p_{n}\left( x,y\right) \leq \frac{C}{V\left( x,e\left( x,n\right) \right) }%
\exp \left( -c\left( \frac{E\left( x,R\right) }{n}\right) ^{\frac{1}{\beta -1%
}}\right) \text{ }  \tag{UE}  \label{LUE1}
\end{equation}%
together with $\left( VD\right) $ and $\left( TC\right) $, where $e(x,n)$ is
the inverse of $E\left( x,R\right) $ in the second variable, $B=B\left(
x,2R\right) .$
\end{theorem}

\begin{corollary}
$\label{cmain}$If $\left( \Gamma ,\mu \right) $ satisfies $\left(
p_{0}\right) ,\left( VD\right) $ and $\left( TC\right) $ then the following
statements are equivalent. (Assuming $\ $that each statement separately
holds for all $x,y\in \Gamma ,R>0,n>0,D\subset A\subset B=B\left(
x,2R\right) $ with fixed $\delta ,C>0,\beta >1$). 
\begin{equation}
\frac{\overline{E}\left( A\right) }{\overline{E}\left( B\right) }\leq
C\left( \frac{\mu \left( A\right) }{\mu \left( B\right) }\right) ^{\delta },
\label{FKE}
\end{equation}%
\begin{equation}
\frac{\lambda ^{-1}\left( A\right) }{\lambda ^{-1}\left( B\right) }\leq
C\left( \frac{\mu \left( A\right) }{\mu \left( B\right) }\right) ^{\delta },
\label{fkll}
\end{equation}%
\begin{equation}
\frac{\rho (D,A)}{\rho \left( x,R,2R\right) }\leq C\left( \frac{\mu \left(
A\right) }{\mu \left( D\right) }\right) ^{\delta }\left( \frac{\mu \left(
D\right) }{\mu \left( B\right) }\right) ^{\delta -1},  \label{FKrr}
\end{equation}%
\begin{equation}
p_{n}(x,x)\leq \frac{C}{V\left( x,e\left( x,n\right) \right) },
\label{duecorr}
\end{equation}%
\begin{equation}
p_{n}\left( x,y\right) \leq \frac{C}{V\left( x,e\left( x,n\right) \right) }%
\exp \left( -c\left( \frac{E\left( x,R\right) }{n}\right) ^{\frac{1}{\beta -1%
}}\right) \text{.}
\end{equation}
\end{corollary}

\begin{remark}
In a related work \cite{Tfull} , among other equivalent conditions the
(elliptic) mean value inequality was used. It says that for all $u$
nonnegative harmonic functions in $B\left( x,R\right) $%
\begin{equation}
u\left( x\right) \leq \frac{C}{V\left( x,R\right) }\sum_{y\in B\left(
x,R\right) }u\left( y\right) \mu \left( y\right) .  \label{mviff}
\end{equation}%
It was shown in \cite{Tfull} that under $\left( p_{0}\right) +\left(
VD\right) +\left( TC\right) $ the mean value inequality is equivalent to the
diagonal upper estimate $\left( DUE\right) $. This means that the mean value
inequality is equivalent to the relative isoperimetric inequalities $\left(
E\right) ,\left( FK\right) ,\left( \rho \right) $ \ and $\left( \ref{FKE}-%
\ref{FKrr}\right) $ provided $\left( VD\right) $ and $\left( TC\right) $
holds. \ In \cite{GT3} a direct proof of $\left( MV\right) \Longrightarrow
\left( FK\right) $ is given for measure metric spaces which works for
weighted graphs as well.
\end{remark}

\section{Basic inequalities}

\setcounter{equation}{0}$\label{sineq}$ In this section basic inequalities
are collected several of them are known, some of them are new.

\begin{lemma}
(c.f. \cite{CG} ) If $\left( p_{0}\right) $ and $(VD)$ hold then for all $%
x\in \Gamma ,R>0,$ $y\in B\left( x,R\right) $ 
\begin{equation}
\frac{V(x,2R)}{V(y,R)}\leq C,  \label{PD2V}
\end{equation}%
furthermore there is an $A_{V}$ such that for all $x\in \Gamma ,R>0$%
\begin{equation}
2V(x,R)\leq V(x,A_{V}R),  \label{PD3V}
\end{equation}%
\begin{equation}
V(x,MR)-V(x,R)\simeq V(x,R)  \label{V3}
\end{equation}%
for any fixed $M\geq 2$, and there is an $\alpha >0$ \ such that for all $%
y\in B\left( x,R\right) ,S\leq R$%
\begin{equation}
\frac{V\left( x,R\right) }{V\left( y,S\right) }\leq C\left( \frac{R}{S}%
\right) ^{\alpha }.  \label{VDaa}
\end{equation}
\end{lemma}

\bigskip The inequality $\left( \ref{PD3V}\right) $ sometimes called
anti-doubling property. As we already mentioned $\left( \ref{VD}\right) $ is
equivalent to $\left( \ref{PD2V}\right) $ and it\ is again evident that both
are equivalent to the inequality 
\begin{equation}
\frac{V(x,R)}{V(y,S)}\leq C\left( \frac{R}{S}\right) ^{\alpha },
\label{vcomp}
\end{equation}%
where $\alpha =\log _{2}D_{V}$ and $d(x,y)<R$. The next Proposition is taken
from \cite{GT1} (see also \cite{Tfull})

\begin{proposition}
\label{Pregvol}\label{plocvol}If $\left( p_{0}\right) $ holds, then for all $%
x,y\in \Gamma $ and $R>0$ and for some $C>1$, 
\begin{equation}
V(x,R)\leq C^{R}\mu (x),  \label{vbound}
\end{equation}%
\begin{equation}
p_{0}^{d(x,y)}\mu (y)\leq \mu (x)  \label{mmccmm}
\end{equation}%
and for any $x\in \Gamma $ 
\begin{equation}
\left\vert \left\{ y:y\sim x\right\} \right\vert \leq \frac{1}{p_{0}}.
\label{deg}
\end{equation}
\end{proposition}

Now we recall some results from \cite{tER} which connect the mean exit time,
the spectral gap, volume and resistance growth.

\begin{theorem}
\label{tallcc}$\left( p_{0}\right) ,\left( VD\right) $ and $\left( TC\right) 
$ \ implies that 
\begin{equation}
\lambda ^{-1}\left( x,2R\right) \asymp E\left( x,2R\right) \asymp \overline{E%
}\left( x,2R\right) \asymp \rho \left( x,R,2R\right) v\left( x,R,2R\right) .
\label{sER}
\end{equation}
\end{theorem}

\begin{theorem}
\label{taDT}For a weighted graph $\left( \Gamma ,\mu \right) $ if%
\begin{equation}
\frac{E\left( x,R\right) }{E\left( y,R\right) }\leq C  \label{wTC}
\end{equation}%
for all $x\in \Gamma ,R\geq 0,y\in B\left( x,R\right) \left( VD\right) $ for
a fixed independent $C>0$ \ then there is an $A_{E}>1$ such that for all $%
x\in \Gamma ,R>0$%
\begin{equation}
E\left( x,A_{E}R\right) \geq E\left( x,R\right) .  \label{aDT}
\end{equation}
\end{theorem}

\begin{remark}
It is immediate from Theorem \ref{taDT} that $\left( TC\right) $ implies $%
\left( \ref{aDT}\right) $, which is the anti-doubling property of the mean
exit time. \newline
It is also shown in \cite{tER} that 
\begin{equation*}
E\left( x,R\right) \geq cR^{2}
\end{equation*}%
provided $\left( p_{0}\right) $ and $\left( VD\right) $ hold. Furthermore $%
E\left( x,R\right) $ for $R\in 
\mathbb{N}
$ is strictly monotone and consequently has inverse 
\begin{equation*}
e\left( x,n\right) =\min \left\{ r\in 
\mathbb{N}
:E\left( x,r\right) \geq n\right\} .
\end{equation*}%
It is worth to recall that\ the following statements are equivalent \newline
1. There are $C,c>0,\beta \geq \beta ^{\prime }>0$ such that for all $x\in
\Gamma ,R\geq S>0,$ $y\in B\left( x,R\right) $%
\begin{equation}
c\left( \frac{R}{S}\right) ^{\beta ^{\prime }}\leq \frac{E\left( x,R\right) 
}{E\left( y,S\right) }\leq C\left( \frac{R}{S}\right) ^{\beta },  \label{Eb}
\end{equation}%
2. There are $C,c>0,\beta \geq \beta ^{\prime }>0$ such that for all $x\in
\Gamma ,n\geq m>0,$ $y\in B\left( x,e\left( x,n\right) \right) $%
\begin{equation}
c\left( \frac{n}{m}\right) ^{1/\beta }\leq \frac{e\left( x,n\right) }{%
e\left( y,m\right) }\leq C\left( \frac{n}{m}\right) ^{1/\beta ^{\prime }}.
\label{eb}
\end{equation}
\end{remark}

\begin{definition}
The local sub-Gaussian kernel is the following$.$\newline
Let $k=k_{z}\left( n,R\right) \geq 0$ the maximal integer for which 
\begin{equation}
\frac{n}{k}\leq qE\left( z,\left\lfloor \frac{R}{k}\right\rfloor \right)
\label{defkernel}
\end{equation}%
and $k_{z}\left( n,R\right) =0$ if there is no such an integer. The
sub-Gaussian kernel is defined as 
\begin{equation*}
k\left( x,n,R\right) =\min_{z\in B\left( x,R\right) }k_{z}\left( n,R\right) .
\end{equation*}
\end{definition}

\begin{remark}
From the definition of $k_{z}\left( n,R\right) $ and $\left( TC\right) $ it
follows easily that%
\begin{equation*}
k_{z}\left( n,R\right) +1\geq c\left( \frac{E\left( z,R\right) }{n}\right) ^{%
\frac{1}{\beta -1}}
\end{equation*}%
and for \thinspace $k\left( x,n,R\right) $ with another use of $\left(
TC\right) $ one obtains that for $d\left( x,z\right) <R$%
\begin{equation}
k\left( x,n,R\right) +1\geq c\left( \frac{E\left( x,R\right) }{n}\right) ^{%
\frac{1}{\beta -1}}.  \label{k>}
\end{equation}
\end{remark}

The equivalence of the isoperimetric inequalities in Theorem \ref{tmain} is
based on the next observation.

\begin{proposition}
\label{pcycle}\label{callcc}Let $\delta >0,A\subset \Gamma $. The following
statements are equivalent. 
\begin{equation}
\overline{E}\left( A\right) \leq C\mu \left( A\right) ^{\delta },
\label{ie1}
\end{equation}%
\begin{equation}
\lambda ^{-1}\left( A\right) \leq C\mu \left( A\right) ^{\delta },
\label{ie2}
\end{equation}%
\begin{equation}
\rho \left( D,A\right) \mu \left( D\right) \leq C\mu \left( A\right)
^{\delta }\text{ for all }D\subset A.  \label{ie3}
\end{equation}
\end{proposition}

The proofs are given via a series of lemmas.

\begin{lemma}
\bigskip \label{llrv}(c.f. Lemma 4.6 \cite{T5}) For all weighted graphs and
for all finite sets, $A\subset B\subset \Gamma $ the inequality 
\begin{equation}
\lambda (B)\rho (A,B^{c})\mu (A)\leq 1,  \label{lrm}
\end{equation}%
holds, in particular%
\begin{equation}
\lambda (x,2R)\rho (x,R,2R)V(x,R)\leq 1.  \label{lrvb}
\end{equation}
\end{lemma}

\begin{lemma}
\label{lriter} (c.f. Proposition 3.2 \cite{T6}) If for a finite $A\subset
\Gamma $ \ there are $C,C^{\prime },\delta >0$ \ such that 
\begin{equation}
\mu \left( D\right) \rho \left( D,A\right) \leq C\mu \left( A\right)
^{\delta }\text{ for all }D\subset A
\end{equation}%
, then 
\begin{equation*}
\overline{E}\left( A\right) \leq C^{\prime }\mu \left( A\right) ^{\delta }.
\end{equation*}
\end{lemma}

\begin{lemma}
\label{lebar} (c.f Lemma 3.6 \cite{TD}) For any finite set $A\subset \Gamma $%
\begin{equation}
\lambda ^{-1}\left( A\right) \leq \overline{E}\left( A\right) .
\label{llebar}
\end{equation}
\end{lemma}

\begin{proof}[Proof of Proposition \protect\ref{pcycle}]
The implication $\left( \ref{ie1}\right) \Longrightarrow \left( \ref{ie2}%
\right) $ \ follows from Lemma \ref{lebar}, $\left( \ref{ie2}\right)
\Longrightarrow \left( \ref{ie3}\right) $ from Lemma \ref{llrv} and finally $%
\left( \ref{ie3}\right) \Longrightarrow \left( \ref{ie1}\right) $ by Lemma %
\ref{lriter}.
\end{proof}

We finish this section showing the connection between the isoperimetric
inequalities in Theorem \ref{tmain} and Corollary \ref{cmain}. \ 

\begin{proposition}
$\label{p3eq}$The statements $\left( \ref{sFKE}\right) ,\left( \ref{sFKll}%
\right) $ \ and $\left( \rho \right) $ are equivalent as well as $\left( \ref%
{FKE}\right) ,\left( \ref{fkll}\right) $ \ and $\left( \ref{FKrr}\right) .$
\ 
\end{proposition}

\begin{proof}
The first statement follows from Proposition \ref{pcycle} setting $%
C=C^{\prime }\frac{E\left( x,R\right) }{V\left( x,R\right) ^{\delta }}$. The
second statement uses Proposition \ref{callcc} and the observation that $%
\left( \ref{FKrr}\right) $ can be written as 
\begin{equation*}
\rho \left( D,A\right) \mu \left( D\right) \leq C\rho \left( x,R,2R\right)
V\left( x,R\right) \frac{\mu \left( A\right) ^{\delta }}{V\left( x,R\right)
^{\delta }}.
\end{equation*}
\end{proof}

\begin{proposition}
\label{pskew}Each statement $\left( \ref{sFKE}\right) ,\left( \ref{sFKll}%
\right) $ \ and $\left( \rho \right) $ implies $\left( VD\right) $ and $%
\left( TC\right) .$
\end{proposition}

\begin{proof}
First let us observe that if one of them implies $\left( VD\right) $ then
all of them do, since they are equivalent by Proposition \ref{p3eq}. \ So we
can choose $\left( \ref{sFKE}\right) .$ Let $A=B\left( x,R\right) $ then $%
A=B\left( y,2R\right) $ we have immediately $\left( VD\right) $ and $\left(
TC\right) .$
\end{proof}

Proposition \ref{pskew} means that the volume doubling property, $\left(
VD\right) $ and the time comparison principle, $\left( TC\right) $ can be
set as precondition in Theorem \ref{tmain} as it is done in Corollary \ref%
{cmain}.

\begin{proposition}
\label{psumm}Theorem \ref{tmain} and Corollary \ref{cmain} mutually imply
each other.
\end{proposition}

\begin{proof}
According to Proposition $\ref{pskew}$ we can set $\left( VD\right) $ and $%
\left( TC\right) $ as preconditions then using Theorem \ref{tallcc} \ the
r.h.s. of each inequality $E\left( x,R\right) $ can be replaced with the
needed term receiving that $\left( \ref{sFKE}\right) \Longrightarrow \left( %
\ref{FKE}\right) ,$ $\left( \ref{sFKll}\right) \Longrightarrow \left( \ref%
{fkll}\right) $ \ and $\left( \rho \right) \Longrightarrow \left( \ref{FKrr}%
\right) .$ The opposite implications can be seen choosing $R^{\prime }=\frac{%
3}{2}R$ and applying $\left( VD\right) ,\left( TC\right) $ and $\left( \ref%
{callcc}\right) .$ This clearly gives the statement. If any of the
isoperimetric inequalities is equivalent to the diagonal upper estimate then
all of them are.
\end{proof}

\section{The upper estimates}

\label{sfk}In this section we shall show the following theorem, which
implies Theorem \ref{tmain} according to Proposition \ref{psumm} and Theorem %
\ref{tallcc}.

\begin{theorem}
\label{tFK}If $\left( \Gamma ,\mu \right) $ satisfies $\left( p_{0}\right)
,\left( VC\right) $ and $\left( TC\right) $ then the following statements
are equivalent 
\begin{equation}
\lambda ^{-1}\left( A\right) \leq CE\left( x,R\right) \left( \frac{\mu
\left( A\right) }{\mu \left( B\right) }\right) ^{\delta }\text{ for all }%
A\subset B\left( x,2R\right) ,  \label{fkmain}
\end{equation}%
\begin{equation}
p_{n}(x,x)\leq \frac{C}{V\left( x,e\left( x,n\right) \right) }.
\label{duemain}
\end{equation}%
\begin{equation}
p_{n}\left( x,y\right) \leq \frac{C}{V\left( x,e\left( x,n\right) \right) }%
\exp \left( -c\left( \frac{E\left( x,d\left( x,y\right) \right) }{n}\right)
^{\frac{1}{\beta -1}}\right) \text{.}  \label{UEmain}
\end{equation}
\end{theorem}

\subsection{Estimate of the Dirichlet heat kernel}

\begin{lemma}
\label{lNash}Let $(\Gamma ,\mu )$ be a weighted graph. \ Assume that for $%
a,C>0$ fixed constants and for any non-empty finite set $A\subset \Gamma $ 
\begin{equation}
\lambda (A)^{-1}\leq aC\mu \left( A\right) ^{\delta }.
\end{equation}%
The for any $f(x)$ non-negative function on $\Gamma $ with finite support 
\begin{equation*}
a\left\Vert f\right\Vert _{2}^{2}\left( \frac{\left\Vert f\right\Vert _{2}}{%
\left\Vert f\right\Vert _{1}}\right) ^{2\delta }\leq C\mathcal{E}\left(
f,f\right) .
\end{equation*}
\end{lemma}

\begin{proof}
The proof is a simple modification of \cite[Lemma 5.2]{GT1} (see also \cite[%
Lemma 2.2]{G1}).
\end{proof}

Now we have to make a careful detour as it was made in \cite{CG} or \cite%
{GT1}. The strategy is the following. \ We consider the weighted graph $%
\Gamma ^{\ast }$ with the same vertex set as $\Gamma $ with new edges and
weights induced by the two-step transition operator $Q=P^{2},$ 
\begin{equation*}
\mu _{x,y}^{\ast }=\mu \left( x\right) P_{2}\left( x,y\right) .
\end{equation*}%
If \ $\Gamma ^{\ast }$ is decomposed into two disconnected components due to
the periodicity of $P$ the applied argument will work irrespective which
component is considered. We show that $\left( p_{0}\right) ,\left( VD\right)
,\left( TC\right) $ \ and $\left( FK\right) $ hold on $\Gamma ^{\ast }$ if
they hold on $\Gamma .$ We deduce the Dirichlet heat kernel estimate for $Q$
\ on $\Gamma ^{\ast }$, then we show that it implies the same on $\Gamma .$
\ We have to do this detour to ensure 
\begin{equation*}
\frac{1}{\mu ^{\prime }\left( x\right) }Q\left( x,x\right) =q\left(
x,x\right) \geq c>0
\end{equation*}%
holds for all $x\in \Gamma ^{\ast }$ which will be needed in the key step to
show the diagonal upper estimate in the proof of Lemma \ref{lprep}.

\begin{lemma}
\label{lfkcsill}If $\left( p_{0}\right) ,\left( VD\right) ,\left( TC\right)
,\left( FK\right) $ and \ holds on $\Gamma $ , then the same is true on $%
\Gamma ^{\ast }$.
\end{lemma}

\begin{proof}
The statement is evident for $\left( p_{0}\right) $ and $\left( VD\right) $.
Here it is worth to mention that $\mu ^{\ast }\left( x\right) =\mu \left(
x\right) $ \ and from $\left( \ref{mmccmm}\right) $ we know that $\mu \left(
x\right) \simeq \mu \left( y\right) $ \ if $x\sim y$. \ Let us observe that 
\begin{eqnarray}
B\left( x,2R\right) &\subset &\overline{B}^{\ast }\left( x,R\right) ,
\label{BBB} \\
B^{\ast }\left( x,R\right) &\subset &B\left( x,2R\right)
\end{eqnarray}%
\ and 
\begin{eqnarray}
V^{\ast }\left( y,2R\right) &\leq &V\left( y,4R\right) \leq C^{2}V\left(
x,R\right)  \label{vv*} \\
&\leq &C^{2}\mu \left( \overline{B}^{\ast }\left( x,R/2\right) \right) \leq
C^{2}V^{\ast }\left( x,R\right) .  \notag
\end{eqnarray}%
Let us note that we have shown that the volumes of the above balls are
comparable.

The next is to show $\left( TC\right) $. 
\begin{eqnarray*}
E^{\ast }\left( y,2R\right) &=&\sum_{z\in B^{\ast }\left( y,2R\right)
}\sum_{k=0}^{\infty }Q_{k}^{B^{\ast }\left( y,2R\right) }\left( y,z\right) \\
&\leq &\sum_{z\in B\left( y,4R\right) }\sum_{k=0}^{\infty }P_{2k}^{B\left(
y,4R\right) }\left( y,z\right) \\
&\leq &\sum_{z\in B\left( y,4R\right) }\sum_{k=0}^{\infty }P_{2k}^{B\left(
y,4R\right) }\left( y,z\right) +P_{2k+1}^{B\left( y,4R\right) }\left(
y,z\right) \\
&=&E\left( y,4R\right) \leq CE\left( x,R/2\right) \\
&=&\sum_{z\in B\left( x,R/2\right) }\sum_{k=0}^{\infty }P_{k}^{B\left(
x,R/2\right) }\left( x,z\right) \\
&=&\sum_{z\in B\left( x,R/2\right) }\sum_{k=0}^{\infty }P_{2k}^{B\left(
x,R/2\right) }\left( x,z\right) +P_{2k+1}^{B\left( x,R/2\right) }\left(
x,z\right) .
\end{eqnarray*}%
Now we use a trivial estimate. 
\begin{eqnarray*}
P_{2k+1}^{B\left( x,R\right) }\left( x,z\right) &=&\sum_{w\sim
z}P_{2k}^{B\left( x,R\right) }\left( x,w\right) P^{B\left( x,R\right)
}\left( w,z\right) \\
&\leq &\sum_{w\sim z}P_{2k}^{B\left( x,R\right) }\left( x,w\right) .
\end{eqnarray*}%
Summing up for all $z$ \ and recalling $\left( \ref{deg}\right) $ which
states that for a fixed $w\in \Gamma $, $\left\vert \left\{ w\sim z\right\}
\right\vert \leq \frac{1}{p_{0}},$ we receive that

\begin{eqnarray*}
\sum_{z\in B\left( x,R/2\right) }P_{2k+1}^{B\left( x,R/2\right) }\left(
x,z\right) &\leq &\sum_{z\in B\left( x,R/2\right) }\sum_{w\sim
z}P_{2k}^{B\left( x,R/2\right) }\left( x,w\right) \\
&\leq &C\sum_{w\in \overline{B}\left( x,R/2\right) }P_{2k}^{B\left(
x,R\right) }\left( x,w\right) .
\end{eqnarray*}%
As a result we obtain that 
\begin{eqnarray*}
E^{\ast }\left( y,2R\right) &\leq &C\sum_{z\in \overline{B}\left(
x,R/2\right) }\sum_{k=0}^{\infty }P_{2k}^{\overline{B}\left( x,R/2\right)
}\left( x,z\right) \\
&\leq &CE^{\ast }\left( x,R/2+1\right) \leq CE^{\ast }\left( x,R\right) .
\end{eqnarray*}%
This shows that $\left( TC\right) $ holds on $\Gamma ^{\ast }$. We have also
proven that 
\begin{equation}
cE^{\ast }\left( x,R\right) \leq E\left( x,R\right) \leq CE^{\ast }\left(
x,R\right) .  \label{ee*}
\end{equation}%
It is left to show that from 
\begin{equation}
\lambda (A)^{-1}\leq CE\left( x,R\right) \left( \frac{\mu \left( A\right) }{%
V\left( x,R\right) }\right) ^{\delta }
\end{equation}%
it follows that 
\begin{equation}
\lambda ^{\ast }(A)^{-1}\leq CE^{\ast }\left( x,R\right) \left( \frac{\mu
^{\ast }\left( A\right) }{V^{\ast }\left( x,R\right) }\right) ^{\delta }
\end{equation}%
holds as well. The inequality 
\begin{equation}
\lambda ^{\ast }\left( A\right) \geq \lambda \left( \overline{A}\right)
\label{l>l}
\end{equation}%
was given in \cite[Lemma 4.3]{CG}. \ Collecting the inequalities we get the
statement. 
\begin{eqnarray*}
\lambda ^{\ast }\left( A\right) ^{-1} &\leq &\lambda \left( \overline{A}%
\right) ^{-1}\leq CE\left( x,R+1\right) \left( \frac{\mu \left( \overline{A}%
\right) }{V\left( x,R+1\right) }\right) ^{\delta } \\
&\leq &CE^{\ast }\left( x,R\right) \left( \frac{\mu ^{\ast }\left( A\right) 
}{V^{\ast }\left( x,R\right) }\right) ^{\delta }
\end{eqnarray*}
\end{proof}

\begin{lemma}
\label{lp<pp}For all random walks on weighted graphs, $x,y\in A\subseteq
\Gamma ,$ $n,m\geq 0$%
\begin{equation}
p_{n+m}^{A}\left( x,y\right) \leq \sqrt{p_{2n}^{A}\left( x,x\right)
p_{2m}^{A}\left( y,y\right) }.  \label{p<pp}
\end{equation}
\end{lemma}

\begin{proof}
The proof is standard, hence omitted.
\end{proof}

To complete the scheme of the proof we need the return from $\Gamma ^{\ast }$
to $\Gamma $. This is given in the following lemma.

\begin{lemma}
\label{lret}Assume that $\left( \Gamma ,\mu \right) $ \ satisfy $\left(
p_{0}\right) \left( VD\right) $ and $\left( TC\right) $. \ In addition if $%
\left( DUE\right) $ holds on $\left( \Gamma ^{\ast },\mu \right) $ , then it
holds $\left( \Gamma ,\mu \right) $.
\end{lemma}

\begin{proof}
The condition states that%
\begin{equation*}
q_{n}\left( x,x\right) \leq \frac{C}{V^{\ast }\left( x,e^{\ast }\left(
x,n\right) \right) }.
\end{equation*}%
Then from the definition of $q$,\ $\left( \ref{vv*}\right) $ and $\left( \ref%
{ee*}\right) $ it follows that%
\begin{equation*}
p_{2n}\left( x,x\right) \leq \frac{C}{V\left( x,e\left( x,2n\right) \right) }%
.
\end{equation*}%
Finally for odd times the statement follows by a standard argument. From the
spectral decomposition of $P_{n}^{B\left( x,R\right) }$ for finite balls one
has that 
\begin{equation*}
P_{2n}^{B\left( x,R\right) }\left( x,x\right) \geq P_{2n+1}^{B\left(
x,R\right) }\left( x,x\right)
\end{equation*}%
and consequently 
\begin{eqnarray*}
p_{2n}\left( x,x\right) &=&\lim_{R\rightarrow \infty }p_{2n}^{B\left(
x,R\right) }\left( x,x\right) \\
&\geq &\lim_{R\rightarrow \infty }p_{2n+1}^{B\left( x,R\right) }\left(
x,x\right) =p_{2n+1}\left( x,x\right) ,
\end{eqnarray*}%
which gives the statement using $\left( VD\right) ,\left( TC\right) $ and $%
\left( \ref{eb}\right) .$
\end{proof}

\begin{lemma}
\label{lprep}If $\left( p_{0}\right) $ is true and $\left( FK\right) :$ 
\begin{equation}
\lambda (A)^{-1}\leq Ca\mu \left( A\right) ^{\delta }
\end{equation}%
holds for 
\begin{equation*}
a=\frac{E\left( x,R\right) }{V\left( x,R\right) ^{\delta }}\simeq \frac{%
E^{\ast }\left( x,R\right) }{V^{\ast }\left( x,R\right) ^{\delta }}
\end{equation*}%
and all $A\subset B^{\ast }\left( x,R\right) $ on $\left( \Gamma ^{\ast
},\mu \right) $, then for all $x,y\in \Gamma $ 
\begin{equation*}
q_{n}^{B^{\ast }\left( x,R\right) }\left( y,y\right) \leq C\left( \frac{a}{n}%
\right) ^{1/\delta }.
\end{equation*}
\end{lemma}

\begin{proof}
The proof is a slight modification of the steps proving $\left( a\right)
\Longrightarrow \left( b\right) $ in Proposition 5.1 of \cite{GT1} so we
omit it. This final statement can be reformulated for $y\in \Gamma $ as
follows 
\begin{equation*}
q_{2n}^{B^{\ast }}\left( y,y\right) \leq \frac{C}{V\left( x,R\right) }\left( 
\frac{E\left( x,R\right) }{n}\right) ^{1/\delta }.
\end{equation*}
\end{proof}

\bigskip

Now we consider the following path decompositions.

\begin{lemma}
\label{lcut}Let $p_{n}\left( x,y\right) $ the heat kernel of the random walk
on an arbitrary weighted graph $\left( \Gamma ,\mu \right) .$ \ Let $%
A\subset \Gamma ,x,y\in A,n>0$ then 
\begin{equation}
p_{n}\left( x,y\right) \leq p_{n}^{A}\left( x,y\right) +P_{x}\left(
T_{A}<n\right) \max\limits_{\substack{ z\in \partial A  \\ 0\leq k<n}}%
p_{k}\left( z,y\right) ,  \label{cut1}
\end{equation}%
\begin{eqnarray}
p_{n}\left( x,y\right) &\leq &p_{n}^{A}\left( x,y\right) +P_{x}\left(
T_{A}<n/2\right) \max\limits_{\substack{ z\in \partial A  \\ n/2\leq k<n}}%
p_{k}\left( z,y\right)  \label{cut2} \\
&&+P_{y}\left( T_{A}<n/2\right) \max\limits_{\substack{ z\in \partial A  \\ %
n/2\leq k<n}}p_{k}\left( z,x\right) .
\end{eqnarray}
\end{lemma}

\begin{proof}
Both inequalities follow, as in \cite[Lemma 2.5]{G1}, from the first exit
decomposition starting from $x$ or from $x$ and $y$ as well. \ 
\end{proof}

\subsection{Proof of the upper estimates}

\begin{proof}[Proof of Theorem \protect\ref{tFK}]
First we show the implication $\left( FK\right) \Longrightarrow \left(
DUE\right) $ on $\Gamma ^{\ast }$ assuming $\left( p_{0}\right) ,\left(
VD\right) $ and $\left( TC\right) $. \ We follow the main lines of \cite{G1}%
. Let we choose $r$ so that $Ln=E\left( x,r\right) $ for a large $L>0$. From 
$\left( \ref{cut2}\right) $ we have that for $B=B^{\ast }\left( x,r\right) $ 
\begin{equation}
q_{n}\left( x,x\right) \leq q_{n}^{B}\left( x,x\right) +2Q_{x}\left(
T_{B}<n/A\right) \max\limits_{\substack{ z\in \partial B  \\ n/A\leq k<n}}%
q_{k}\left( z,x\right) .  \label{prek1}
\end{equation}%
From $\left( \ref{p<pp}\right) $ one gets that for all $n/A\leq k<n$%
\begin{equation*}
q_{k}\left( z,x\right) \leq \sqrt{q_{k}\left( z,z\right) q_{k}\left(
x,x\right) }\leq \max\limits_{v\in \overline{B}}q_{k}\left( v,v\right) \leq
C_{1}\max\limits_{v\in \overline{B}}q_{\left\lfloor n/A\right\rfloor }\left(
v,v\right) .
\end{equation*}%
This results in $\left( \ref{prek1}\right) $ that for some $x_{1}\in 
\overline{B}$ 
\begin{equation}
q_{n}\left( x,x\right) \leq q_{n}^{B}\left( x,x\right) +2Q_{x}\left(
T_{B}<n/A\right) C_{1}q_{\left\lfloor n/A\right\rfloor }\left(
x_{1},x_{1}\right) .  \label{p1}
\end{equation}%
We continue this procedure. \ In the $i$-th step we have 
\begin{equation}
q_{n_{i-1}}\left( x_{i},x_{i}\right) \leq q_{n_{i}}^{B_{i}}\left(
x_{i},x_{i}\right) +2Q_{x_{i}}\left( T_{B_{i}}<n_{i+1}\right)
C_{1}q_{n_{i}}\left( x_{i+1},x_{i+1}\right) ,  \label{pi}
\end{equation}%
where $n_{i}=\left\lfloor n/A^{i}\right\rfloor ,r_{i}=e\left(
x_{i},Ln_{i}\right) ,B_{i}=B\left( x_{i},r_{i}\right) ,x_{i+1}\in \overline{B%
}_{i}$. Let $m=\left\lfloor \log _{A}n\right\rfloor $ and we stop the
iteration at $m$. This means that $1\leq \left\lfloor \frac{n}{A^{m}}%
\right\rfloor =n_{m}<A$. Now we choose $A$. From the definition of $n_{i}$
and from $\left( TC\right) $ it follows that%
\begin{equation*}
A=\frac{Ln_{i}}{Ln_{i+1}}=\frac{E\left( x_{i,}r_{i}\right) }{E\left(
x_{i+1},r_{i+1}\right) }\leq \frac{E\left( x_{i,}2r_{i}\right) }{E\left(
x_{i+1},r_{i+1}\right) }\leq C\left( \frac{2r_{i}}{r_{i+1}}\right) ^{\beta },
\end{equation*}%
which results by $\sigma =2\left( \frac{C}{A}\right) ^{\frac{1}{\beta }}<1/2$
that 
\begin{equation}
r_{i+1}\leq \sigma r_{i}  \label{rdecr}
\end{equation}%
if $A>4^{\beta }C$. \ From the choice of the constants one obtains that 
\begin{equation}
d\left( x,x_{m}\right) \leq r+r_{1}+...+r_{m}\leq r\sum_{i=0}^{m}\sigma
^{i}<r\frac{1}{1-\sigma }\leq 2r.  \label{dist}
\end{equation}%
From Lemma \ref{lprep} the first term can be estimated as follows 
\begin{eqnarray*}
q_{n_{i}}^{B_{i}}\left( x_{i},x_{i}\right) &\leq &\frac{C}{V\left(
x_{i},r_{i}\right) }\left( \frac{E\left( x,r_{i}\right) }{n_{i}}\right)
^{1/\delta } \\
&=&\frac{C}{V\left( x_{i},r_{i}\right) }L^{1/\delta }=\frac{CL^{1/\delta }}{%
V\left( x,2r\right) }\frac{V\left( x,2r\right) }{V\left( x_{i},r_{i}\right) }
\\
&\leq &\frac{CL^{1/\delta }}{V\left( x,r\right) }\left( \frac{2}{\sigma }%
\right) ^{\alpha i},
\end{eqnarray*}%
where in the last step $\left( \ref{dist}\right) $ and $\left( VD\right) $
have been used. Let us observe that by definition of $k=k\left(
x_{i},n_{i+1},r_{i}\right) $ 
\begin{equation*}
\frac{n_{i+1}}{k+1}>\min\limits_{y\in B_{i}}E\left( y,\frac{r_{i}}{k+1}%
\right)
\end{equation*}%
and by $\left( TC\right) $ 
\begin{eqnarray*}
Ln_{i} &=&E\left( x_{i},r_{i}\right) \leq CE\left( y,r_{i}\right) \leq
C\left( k+1\right) ^{\beta }E\left( y,\frac{r_{i}}{k+1}\right) \\
&\leq &C\left( k+1\right) ^{\beta -1}n_{i},
\end{eqnarray*}%
which results that $k>\left( L/C\right) ^{\frac{1}{\beta -1}}-1$ and 
\begin{equation*}
Q_{x_{i}}\left( T_{B}<n_{i+1}\right) \leq C\exp \left[ -ck\left(
x_{i},n_{i+1},r_{i}\right) \right] \leq C\exp \left[ -c\left( \frac{L}{C}%
\right) ^{\frac{1}{\beta -1}}\right] .
\end{equation*}%
This means that 
\begin{equation*}
Q_{x_{i}}\left( T_{B}<n_{i+1}\right) \leq \frac{\varepsilon }{2}
\end{equation*}%
if $L$ is chosen to be enough large. \ Inserting this into $\left( \ref{pi}%
\right) $ one obtains 
\begin{equation}
q_{n_{i-1}}\left( x_{i},x_{i}\right) \leq \frac{CL^{1/\delta }}{V\left(
x,r\right) }\left( \frac{2}{\sigma }\right) ^{\alpha i}+\varepsilon
C_{1}q_{n_{i}}\left( x_{i+1},x_{i+1}\right) .
\end{equation}%
Summing up the iteration results in%
\begin{equation}
q_{n}\left( x,x\right) \leq \frac{CL^{1/\delta }}{V\left( x,r\right) }%
\sum_{i=1}^{m}\left( \left( \frac{2}{\sigma }\right) ^{\alpha }\varepsilon
\right) ^{i}+\left( \varepsilon C_{1}\right) ^{m}q_{m}\left(
x_{m},x_{m}\right) .  \label{p<3}
\end{equation}%
Choosing $L$ enough large $\varepsilon <\min \left\{ \left( \frac{\sigma }{2}%
\right) ^{\alpha },\frac{1}{C_{1}}\right\} $ can be ensured. This means that
the sum in the first term is bounded by $1/\left( 1-\varepsilon \left( \frac{%
2}{\sigma }\right) ^{\alpha }\right) <C$. The second term can be treated as
follows. 
\begin{equation*}
q_{m}\left( x_{m},x_{m}\right) =\frac{1}{\mu \left( x_{m}\right) }%
Q_{m}\left( x_{m},x_{m}\right) \leq \frac{1}{\mu \left( x_{m}\right) }.
\end{equation*}%
From $\left( \ref{rdecr}\right) $ we have that 
\begin{eqnarray*}
\frac{1}{\mu \left( x_{m}\right) } &=&\frac{1}{V\left( x,r\right) }\frac{%
V\left( x,2r\right) }{\mu \left( x_{m}\right) } \\
&\leq &\frac{1}{V\left( x,r\right) }\left( 2r\right) ^{\alpha }.
\end{eqnarray*}%
Consequently we are ready if 
\begin{equation*}
\left( 2r\right) ^{\alpha }\varepsilon ^{m}<C^{\prime }.
\end{equation*}%
Let us remark that $E\left( z,r\right) \geq r$. which implies that $e\left(
x,n\right) \leq n.$ \ From the definition of $m$ and $E\left( x,r\right)
=Ln, $ 
\begin{equation*}
\left( 2r\right) ^{\alpha }\varepsilon ^{m}\leq \left( 2r\right) ^{\alpha
}\varepsilon ^{\log _{A}n}\leq \left[ 2E\left( x,r\right) \right] ^{\alpha
}n^{\log _{A}\varepsilon }=\left( 2L\right) ^{\alpha }n^{\alpha +\log
_{A}\varepsilon }\leq C
\end{equation*}%
if $\varepsilon <A^{-\alpha }$, $L$ is enough large. \ Finally from $\left( %
\ref{p<3}\right) $ we receive that 
\begin{eqnarray}
q_{n}\left( x,x\right) &\leq &\frac{CL^{1/\delta }}{V\left( x,r\right) }%
\sum_{i=1}^{m}\left( 2^{\alpha }\varepsilon \right) ^{i}+\left( \varepsilon
C_{1}\right) ^{m}q_{n_{m}}\left( x_{m},x_{m}\right) \\
&\leq &\frac{C}{V\left( x,r\right) }=\frac{C}{V\left( x,e\left( x,Ln\right)
\right) }\leq \frac{C}{V\left( x,e\left( x,n\right) \right) },
\end{eqnarray}%
if $\varepsilon <\min \left\{ \left( \frac{\sigma }{2}\right) ^{\alpha },%
\frac{1}{C_{1}},A^{-\alpha }\right\} ,$ absorbing all the constants into $C$%
. This means that $\left( DUE\right) $ holds on $\Gamma ^{\ast }$ and by
Lemma $\left( \ref{lret}\right) $ $\left( DUE\right) $ holds on $\Gamma $ as
well. It was shown in \cite{Tfull} that under the assumption $\left(
p_{0}\right) $%
\begin{equation*}
\left( VD\right) +\left( TC\right) +\left( DUE\right) \Longrightarrow \left(
UE\right) .
\end{equation*}%
The reverse implication $\left( UE\right) \Longrightarrow \left( DUE\right) $
is evident, hence the proof of Theorem \ref{tFK} and \ref{tmain} is
complete. Let us mention that the implication $\left( DUE\right)
\Longrightarrow \left( FK\right) $ can be seen as it was given in \cite{CG}
without any essential change and the proof of Theorem \ref{tmain} and \ref%
{tFK} is complete.
\end{proof}

\subsection{A Davies-Gaffney type inequality}

\label{sDI}

We provide here a different proof of the upper estimate which might be
interesting on its own. The proof has two ingredients. \ The first one is
the generalization of the Davies-Gaffney inequality. First we need a theorem
from \cite{TD}.

\begin{theorem}
\label{tebarimp} If $\left( p_{0}\right) $ and $\left( \overline{E}\right) $
hold then there are $c,C>0$ such that for all $x\in \Gamma ,n,R>0$%
\begin{equation}
\mathbb{P}_{x}(T_{x,R}<n)\leq C\exp \left[ -ck\left( x,n,R\right) \right] .
\label{lpsi}
\end{equation}
\end{theorem}

\begin{proof}[Proof of Theorem \protect\ref{tebarimp}]
See Theorem 5.1 \cite{TD}.
\end{proof}

\begin{notation}
Denote 
\begin{equation}
k\left( n,A,B\right) =\min\limits_{z\in A}k\left( z,n,d\right) ,
\label{kapdef1}
\end{equation}%
\ where $d=d\left( A,B\right) $ and 
\begin{equation}
\kappa \left( n,A,B\right) =\max \left\{ k\left( n,A,B\right) ,k\left(
n,B,A\right) \right\} .  \label{kapdef2}
\end{equation}
\end{notation}

\begin{theorem}
\label{tGDI}If $\left( \overline{E}\right) $ holds for a reversible Markov
chain then there is a constant $c>0$ such that for all $A,B\subset V,$ the
Davies-Gaffney type inequality $\left( DG\right) $%
\begin{equation}
\sum_{x\in A,y\in B}p_{n}(x,y)\mu (x)\mu \left( y\right) \leq \left[ \mu
\left( A\right) \mu \left( B\right) \right] ^{1/2}\exp \left( -c\kappa
\left( n,A,B\right) \right)  \tag{DG}  \label{D}
\end{equation}%
holds.
\end{theorem}

\begin{proof}
Using the Chebysev inequality one receives 
\begin{eqnarray}
&&\sum_{x\in A,y\in B}P_{n}(x,y)\mu (x)  \label{gdi1} \\
&=&\sum_{x\in \Gamma }\mu \left( x\right) ^{1/2}I_{A}\left( x\right) \left[
\mu ^{1/2}(x)I_{A}\left( x\right) \sum_{y\in B}P_{n}(x,y)I_{B}\left(
y\right) \right] \\
&\leq &\left( \mu \left( A\right) \right) ^{1/2}\left\{ \sum_{x\in \Gamma
}\mu \left( x\right) I_{A}\left( x\right) \left[ \sum_{y\in \Gamma
}P_{n}\left( x,y\right) I_{B}\left( y\right) \right] ^{2}\right\} ^{1/2}. 
\notag
\end{eqnarray}%
Let us deal with the second term denoting $r=d\left( A,B\right) $ 
\begin{eqnarray}
&&\sum_{x\in \Gamma }\mu \left( x\right) I_{A}\left( x\right) \left[
\sum_{y\in \Gamma }P_{n}\left( x,y\right) I_{B}\left( y\right) \right] ^{2}
\label{gdi2} \\
&=&\sum_{x\in \Gamma }\mu \left( x\right) I_{A}\left( x\right) \sum_{y\in
\Gamma }P_{n}\left( x,y\right) I_{B}\left( y\right) \sum_{z\in \Gamma
}P_{n}\left( x,z\right) I_{B}\left( z\right)  \notag \\
&=&\sum_{x\in \Gamma }\sum_{y\in \Gamma }\sum_{z\in \Gamma }P_{n}\left(
x,z\right) I_{B}\left( z\right) \mu \left( x\right) I_{A}\left( x\right)
P_{n}\left( x,y\right) I_{B}\left( y\right) \\
&=&\sum_{z\in \Gamma }\sum_{y\in \Gamma }\sum_{x\in \Gamma }P_{n}\left(
z,x\right) I_{B}\left( z\right) \mu \left( z\right) I_{A}\left( x\right)
P_{n}\left( x,y\right) I_{B}\left( y\right)  \label{last} \\
&\leq &\sum_{z\in \Gamma }\sum_{x\in \Gamma }P_{n}\left( z,x\right)
I_{B}\left( z\right) \mu \left( z\right) I_{A}\left( x\right) \sum_{y\in
\Gamma }P_{n}\left( x,y\right) I_{B}\left( y\right)  \notag \\
&\leq &\sum_{z\in \Gamma }\sum_{x\in \Gamma }P_{n}\left( z,x\right)
I_{B}\left( z\right) \mu \left( z\right) I_{A}\left( x\right) \\
&\leq &\sum_{z\in \Gamma }P_{n}\left( z,A\right) I_{B}\left( z\right) \mu
\left( z\right) \leq \sum_{z\in \Gamma }P_{z}(T_{z,r}<n)I_{B}\left( z\right)
\mu \left( z\right)  \notag \\
&\leq &\max\limits_{z\in B}\exp \left[ -ck\left( z,n,r\right) \right] \mu
\left( B\right) .
\end{eqnarray}%
The combination of $\left( \ref{gdi1}\right) $ and $\left( \ref{gdi2}-\ref%
{last}\right) $ gives the second term in the definition of $\kappa $ and by
symmetry one can obtain the first one.
\end{proof}

\subsection{The parabolic mean value inequality}

\label{ssPMV}In order to show the off-diagonal upper estimate we need that
the so called parabolic mean value $\left( PMV\right) $ inequality follows
from the diagonal upper estimate. Working under the conditions $\left(
p_{0}\right) ,\left( VD\right) $ \ and $\left( TC\right) $ \ we will show
the following implications 
\begin{equation*}
\left( DUE\right) \Longrightarrow \left( PMV\right)
\end{equation*}

and%
\begin{equation*}
\left( PMV\right) +\left( DG\right) \Longrightarrow \left( UE\right) .
\end{equation*}%
In doing so we introduce $\left( PMV\right) $.

\begin{definition}
A weighted graph satisfies the parabolic mean value inequality\textbf{\ }$%
\left( \mathbf{PMV}\right) $ if for fixed constants $0<c_{1}<c_{2}$ there is
a $C>1$ such that for arbitrary $x\in \Gamma $ \ and $R>0,$ using the
notations $\ E=E\left( x,R\right) ,B=B\left( x,R\right) ,n=c_{2}E,\Psi =%
\left[ 0,n\right] \times B$ for any non-negative Dirichlet solution of the
heat equation 
\begin{equation*}
P^{B}u_{i}=u_{i+1}
\end{equation*}%
on $\Psi ,$ the inequality 
\begin{equation}
u_{n}(x)\leq \frac{C}{E\left( x,R\right) V(x,R)}\sum_{i=c_{1}E}^{c_{2}E}%
\sum_{y\in B(x,R)}u_{i}(y)\mu (y)\,  \label{PMV}
\end{equation}%
holds.
\end{definition}

\begin{theorem}
\label{tDUE->PMV}If $\left( \Gamma ,\mu \right) $ \ satisfies $\left(
p_{0}\right) ,\left( VD\right) $ \ and $\left( TC\right) $, then 
\begin{equation*}
\left( DUE\right) \Longrightarrow \left( PMV\right)
\end{equation*}
\end{theorem}

\begin{proof}
For the proof see \cite{Tfull}.
\end{proof}

\begin{remark}
Let us observe that if for non-negative Dirichlet (sub-)solutions the
parabolic mean value inequality holds then it holds on non-negative
(sub-)solutions as well. \ This can be seen by the decomposition of an $%
u\geq 0$ solution on $B\left( x,2R\right) $ on the smaller ball $B\left(
x,R\right) $ into nonnegative combination of non-negative Dirichlet
solutions in $B\left( x,2R\right) .$ (c.f. \cite{De}).
\end{remark}

\subsection{The local upper estimates}

\setcounter{equation}{0}\label{sLUE}

\begin{proposition}
\label{tLUEvv}Assume that $\left( \Gamma ,\mu \right) $ satisfies $\left(
p_{0}\right) $, $\left( PMV\right) $ and $\left( TC\right) $. Let $x,y\in
\Gamma $ then there are $c,C>0,\beta >1$ such that for all $x,y\in \Gamma
,n>0$%
\begin{equation}
p_{n}\left( x,y\right) \leq \frac{C}{\sqrt{V\left( x,e\left( x,n\right)
\right) V\left( y,e\left( y,n\right) \right) }}\exp \left[ -c\left( \frac{%
E\left( x,d\left( x,y\right) \right) }{n}\right) ^{\frac{1}{\beta -1}}\right]
.  \label{interLUE}
\end{equation}
\end{proposition}

\begin{proof}
The proof \ combines the repeated use of the parabolic mean value inequality
and the Davies-Gaffney inequality. We follow the idea of \cite{LW}. \ Denote 
$R=e\left( x,n\right) ,S=e\left( y,n\right) $ and assume that $d\geq \frac{2%
}{3}\left( R+S\right) $ which ensures that $r=d-R-S\geq \frac{1}{3}d$. From $%
\left( PMV\right) $ it follows that \ 
\begin{equation*}
p_{n}\left( x,y\right) \leq \frac{C}{V\left( x,R\right) E\left( x,R\right) }%
\sum_{c_{1}E\left( x,R\right) }^{c_{2E\left( x,R\right) }}\sum_{z\in V\left(
x,R\right) }p_{k}\left( z,y\right) \mu \left( z\right)
\end{equation*}%
and using $\left( PMV\right) $ for $p_{k}\left( z,y\right) $ on gets 
\begin{equation}
p_{n}\left( x,y\right) \leq \frac{C}{V\left( x,R\right) V\left( y,S\right)
n^{2}}\sum_{i=c_{1}n}^{c_{2}n}\sum_{z\in V\left( x,R\right)
}\sum_{j=c_{1}n+i}^{c_{2}n+i}\sum_{w\in V\left( y,S\right) }p_{j}\left(
z,w\right) \mu \left( z\right) \mu \left( w\right)  \label{pinterm1}
\end{equation}%
Now by $\left( \ref{D}\right) $ and $\left( \ref{k>}\right) $ and denoting $%
A=B\left( x,R\right) ,B=B\left( y,S\right) $ we obtain 
\begin{equation}
p_{n}\left( x,y\right) \leq \frac{C\sqrt{V\left( x,R\right) V\left(
y,S\right) }}{V\left( x,R\right) V\left( y,S\right) n^{2}}%
\sum_{i=c_{1}n}^{c_{2}n}\sum_{j=c_{1}n+i}^{c_{2}n+i}e^{-c\kappa \left(
n,A,B\right) }.  \label{pinterm2}
\end{equation}%
Using $\left( TC\right) $ and $R<\frac{3}{2}d$ one can see that%
\begin{equation*}
\max_{z\in V\left( x,R\right) }\exp -c\left( \frac{E\left( z,d/3\right) }{n}%
\right) ^{\frac{1}{\beta -1}}\leq \exp -c\left( \frac{E\left( x,d\right) }{n}%
\right) ^{\frac{1}{\beta -1}}
\end{equation*}%
and similarly%
\begin{eqnarray*}
\max_{w\in V\left( y,R\right) }\exp -c\left( \frac{E\left( w,d/3\right) }{n}%
\right) ^{\frac{1}{\beta -1}} &\leq &\exp -c\left( \frac{E\left( y,d\right) 
}{n}\right) ^{\frac{1}{\beta -1}} \\
&\leq &\exp -c\left( \frac{E\left( x,d\right) }{n}\right) ^{\frac{1}{\beta -1%
}},
\end{eqnarray*}%
which results that 
\begin{equation*}
p_{n}\left( x,y\right) \leq \frac{C}{\sqrt{V\left( x,R\right) V\left(
y,S\right) }}\exp \left[ -c\left( \frac{E\left( x,d\right) }{n}\right) ^{%
\frac{1}{\beta -1}}\right] .
\end{equation*}%
It is left to treat the case $d\left( x,y\right) <\frac{2}{3}\left(
R+S\right) $. In this case $\kappa \left( n,B\left( x,R\right) ,B\left(
y,S\right) \right) =0$ in $\left( \ref{pinterm2}\right) .$ On the other hand
either $d\left( x,y\right) <R$ or $d\left( x,y\right) <S$ 
\begin{equation*}
E\left( x,d\right) \leq E\left( x,e\left( x,n\right) \right) =n,
\end{equation*}%
which results that $1\leq C\exp \left[ -c\left( \frac{E\left( x,d\right) }{n}%
\right) ^{\frac{1}{\beta -1}}\right] $ for a fixed $C>0$ and similarly if $%
d\left( x,y\right) <S,$ $1\leq C\exp \left[ -c\left( \frac{E\left(
y,d\right) }{n}\right) ^{\frac{1}{\beta -1}}\right] \leq C\exp \left[
-c\left( \frac{E\left( x,d\right) }{n}\right) ^{\frac{1}{\beta -1}}\right] $
which gives the statement.
\end{proof}

The next lemma is from \cite{Tfull}, which leads to the upper estimate.

\begin{lemma}
\label{lvv}If $\left( p_{0}\right) ,\left( VD\right) $ and $\left( TC\right) 
$ hold then for all $\varepsilon >0$ there are $C_{\varepsilon },C>0$ such
that for all $n>0,x,y\in \Gamma ,d=d\left( x,y\right) $%
\begin{equation*}
\sqrt{\frac{V\left( x,e\left( x,n\right) \right) }{V\left( y,e\left(
y,n\right) \right) }}\leq C_{\varepsilon }\exp \varepsilon C\left( \frac{%
E\left( x,d\right) }{n}\right) ^{\frac{1}{\left( \beta -1\right) }}.
\end{equation*}
\end{lemma}

\begin{theorem}
\label{tLUE2}Assume that $\left( \Gamma ,\mu \right) $ satisfies $\left(
p_{0}\right) $, $\left( VD\right) ,\left( TC\right) $ and $\left( DUE\right) 
$. Let $x,y\in \Gamma $, then $\left( UE\right) $ holds: 
\begin{equation}
p_{n}\left( x,y\right) \leq \frac{C}{V\left( x,e\left( x,n\right) \right) }%
\exp \left[ -c\left( \frac{E\left( x,d\left( x,y\right) \right) }{n}\right)
^{\frac{1}{\beta -1}}\right] .
\end{equation}
\end{theorem}

\begin{proof}
From Theorem \ref{tDUE->PMV} we have that%
\begin{equation*}
\left( DUE\right) \Longrightarrow \left( PMV\right) .
\end{equation*}%
Now we can use Proposition \ref{tLUEvv} \ which states that from $\left(
PMV\right) $ and $\left( TC\right) $ it follows that 
\begin{equation*}
p_{n}\left( x,y\right) \leq \frac{C}{\sqrt{V\left( x,e\left( x,n\right)
\right) V\left( y,e\left( y,n\right) \right) }}\exp \left[ -c\left( \frac{%
E\left( x,d\left( x,y\right) \right) }{n}\right) ^{\frac{1}{\beta -1}}\right]
.
\end{equation*}%
Let us use Lemma \ref{lvv}, 
\begin{eqnarray*}
p_{n}\left( x,y\right) &\leq &\frac{C}{V\left( x,e\left( x,n\right) \right) }%
\sqrt{\frac{V\left( x,e\left( x,n\right) \right) }{V\left( y,e\left(
y,n\right) \right) }}\exp \left[ -c\left( \frac{E\left( x,d\left( x,y\right)
\right) }{n}\right) ^{\frac{1}{\beta -1}}\right] \\
&\leq &\frac{CC_{\varepsilon }}{V\left( x,e\left( x,n\right) \right) }\exp %
\left[ \varepsilon C\left( \frac{E\left( x,r\right) }{n}\right) ^{\frac{1}{%
\left( \beta -1\right) }}-c\left( \frac{E\left( x,d\left( x,y\right) \right) 
}{n}\right) ^{\frac{1}{\beta -1}}\right]
\end{eqnarray*}%
and choosing $\varepsilon $ small enough we get the statement.
\end{proof}

\section{Example}

\setcounter{equation}{0}\label{sexample}In this section we recall from \cite%
{Tfull} an example for a graph which is not covered by any of the previous
results of on- and off-diagonal upper estimates but satisfies the conditions
of Theorem \ref{tmain}.

Let $G_{i}$ be the subgraph of the Vicsek tree  (c.f. \cite%
{GT2}\ ) which contains the root $\ z_{0}$ and has diameter $D_{i}=23^{i}$.
Let us denote by $z_{i}$ the cutting points on the infinite path with $%
d\left( z_{0},z_{i}\right) =D_{i}$. \ Denote $G_{i}^{\prime }=\left(
G_{i}\backslash G_{i-1}\right) \cup \left\{ z_{i-1}\right\} $ for $i>0,$ the
annulus defined by $G$-s.

The new graph is defined by stretching the Vicsek tree as follows. \
Consider the subgraphs $G_{i}^{\prime }$ and replace all the edges of them
by a path of length $i+1$. Denote the new subgraph by $A_{i},$ the new
blocks by $\Gamma _{i}=\cup _{j=0}^{i}A_{i},$ the new graph is $\Gamma =\cup
_{j=0}^{\infty }A_{j}$. We denote again by $z_{i}$ the cutting point between 
$A_{i}$ \ and $A_{i-1}$.\ For $x\neq y,x\sim y$ let $\mu _{x,y}=1$.

One can see that neither the volume nor the mean exit time grows
polynomially on $\Gamma ,$ on the other hand, $\Gamma $\ is a tree and the
resistance grows asymptotically linearly on it.

It was shown in \cite{Tfull} that the tree $\Gamma $ \ satisfies $\left(
p_{0}\right) ,\left( VD\right) ,\left( TC\right) $ furthermore the mean
value inequality (for all the definitions and details see \cite{Tfull}). \
The main result there states that under these conditions the diagonal upper
estimate holds. Since $\Gamma $ satisfies $\left( p_{0}\right) ,\left(
VD\right) ,\left( TC\right) $ and $\left( DUE\right) $ we are in the scope
of Theorem \ref{tmain} and all the isoperimetric inequalities hold.

\end{document}